\documentclass[10pt, leqno]{article}

\usepackage{amsmath,amssymb,amsthm, epsfig}

\usepackage{hyperref}

\usepackage{amsmath}

\usepackage{amssymb}

\usepackage{hyperref}

\usepackage{color}

\usepackage{ulem}

\usepackage{epsfig}

\usepackage[mathscr]{eucal}

\usepackage{mathptmx} 
\usepackage[scaled=0.90]{helvet} 
\usepackage{courier} 
\normalfont
\usepackage[T1]{fontenc}

\title{Optimal gradient continuity  for \\
 degenerate elliptic equations}
\author{\it by \smallskip \\
Dami\~ao J.  Ara\'ujo, \quad Gleydson Ricarte,  \quad   Eduardo V. Teixeira \footnote{Corresponding author.}}

\date{}

\newlength{\hchng}
\newlength{\vchng}
\def \p {\vec{{\mathrm{p}}}}
\def \q {\vec{{\mathrm{q}}}}

\newtheorem{theorem}{Theorem}[section]

\newtheorem{lemma}[theorem]{Lemma}

\newtheorem{proposition}[theorem]{Proposition}

\newtheorem{corollary}[theorem]{Corollary}

\theoremstyle{definition}

\theoremstyle{remark}

\numberwithin{equation}{section}

\newcommand{\intav}[1]{\mathchoice {\mathop{\vrule width 6pt height 3 pt depth  -2.5pt
\kern -8pt \intop}\nolimits_{\kern -6pt#1}} {\mathop{\vrule width
5pt height 3  pt depth -2.6pt \kern -6pt \intop}\nolimits_{#1}}
{\mathop{\vrule width 5pt height 3 pt depth -2.6pt \kern -6pt
\intop}\nolimits_{#1}} {\mathop{\vrule width 5pt height 3 pt depth
-2.6pt \kern -6pt \intop}\nolimits_{#1}}}

\begin{document}
\maketitle

\begin{abstract}
We establish new, optimal gradient continuity estimates for solutions to a class of 2nd order partial differential equations, $\mathscr{L}(X, \nabla u, D^2 u) = f$,  whose diffusion properties (ellipticity) degenerate along the \textit{a priori} unknown singular set of an existing solution, $\mathscr{S}(u) := \{ X : \nabla u(X) = 0 \}$. The innovative feature of our main result concerns its optimality -- the sharp, encoded smoothness aftereffects of the operator. Such a quantitative information usually plays a decisive role in the analysis of a number of analytic and geometric problems. Our result is new even for the classical equation $|\nabla u | \cdot \Delta u = 1$. We further apply these new estimates in the study of some well known problems in the theory of elliptic PDEs. 

\medskip


\medskip



\end{abstract}

\section{Introduction}

Regularity theory for solutions to partial differential equations has been a central subject of research since the foundation of the modern analysF of PDEs, back in the 18th century. Of particular interest  are physical and social phenomena that involve diffusion processes, whose mathematical models are governed by second order elliptic PDEs.

\par

\medskip

Smoothness of weak solutions to 2nd order uniformly elliptic equations, both in divergence and in non-divergence forms, is nowadays fairly well established. The cornerstone of the theory is a \textit{universal} modulus of continuity for solutions to the homogeneous equation: $Lu = 0$. This is the contents of DeGiorgi-Nash-Moser theory for the divergence equations and Krylov-Safonov Harnack inequality for non-divergence operators.   

\par

\medskip

Despite of the profound importance of the supra-cited works, a large number of mathematical models involve operators whose ellipticity degenerates along an \textit{a priori} unknown region, that might depend on the solution itself: the free boundary of the problem. This fact impels less efficient diffusion features for the model near such a region and therefore the regularity theory for solutions to such equations become more sophisticated from the mathematical view point.

\par

\medskip

The most typical case of elliptic degeneracy occurs along the singular set of an existing solution:
$$
	\mathscr{S}(u) := \{ X : \nabla u(X) = 0 \}.
$$ 
In fact, a number of degenerate elliptic equations  has its degree of degeneracy comparable to
\begin{equation}\label{model1}
	f(\nabla u) |D^2 u| \approx 1,
\end{equation}
for some function $f\colon \mathbb{R}^d \to \mathbb{R}$, with $\text{Zero}(f) = \{0\}$.  Thus, understanding the precise effect on the lack of smoothness impelled by the emblematic model \eqref{model1} shades lights on  the underlying sharp regularity theory for a number of typical degenerate elliptic operators -- see the heuristic comments in Section \ref{sec statement}.

\par

\medskip

The main goal of this present work is to derive sharp interior regularity estimates for degenerate elliptic equations of the general form
\begin{equation}\label{Eq-Intro}
\mathcal{H}(X,\nabla u)F(X,D^2u)=f(X) , \quad B_1 \subset \mathbb{R}^d,
\end{equation}
where $f \in L^\infty(B_1)$ and  $\mathcal{H} \colon B_{1} \times \mathbb{R}^d \rightarrow \mathbb{R}$  degenerates as
\begin{equation}\label{(2)}
\lambda |\p|^{\gamma} \le \mathcal{H}(X,\p) \le \Lambda |\p|^{\gamma},
\end{equation}
for some  $\gamma >0$.  The $2^{\text{nd}}$ order operator $F\colon B_1\times \text{Sym}(d) \to \mathbb{R}$ in equation \eqref{Eq-Intro} is responsible for diffusion, i.e.,  $F$ will be assumed to be a generic fully nonlinear uniformly elliptic operator:  $\lambda \text{Id}_{d\times d} \le \partial_{i,j} F(X,M) \le \Lambda \lambda  \text{Id}_{d\times d}$. In Section \ref{sec preliminar} we give a more appropriate notion of ellipticity. 

\par

\medskip

Regularity theory for viscosity solutions to fully nonlinear uniformly elliptic equations, 
$$
	F(D^2u) = 0,
$$ 
has attracted the attention of the mathematical community for the last three decades or so. It is well established that solutions to the homogeneous equation is locally of class $C^{1,{\alpha_0}}$ for a universal exponent $\alpha_0$, i.e., depending only on $d$, $\lambda$ and $\Lambda$, see for instance \cite{CC}. If no additional structure is imposed on $F$, $C^{1,\alpha_0}$ is in fact optimal, see \cite{NV1}, \cite{NV2}, \cite{NV}.

\par

\medskip

A quick inference on the structure of equation \eqref{Eq-Intro} reveals that no universal regularity theory for such equation could go beyond $C^{1,\alpha_0}$. In fact the degeneracy term  $ \mathcal{H}(X,\nabla u)$  forces solutions to be less regular than solutions to the uniformly elliptic problem near its singular set.   This particular feature indicates that obtaining sharp regularity estimates for solutions to \eqref{Eq-Intro} should not follow from perturbation techniques. Indeed, it requires new ideas involving an interplay balance between the universal regularity theory for uniform elliptic equations and the degeneracy effect on the diffusion attributes of the operator coming from \eqref{(2)}.

\par

\medskip

In this present work we show that a viscosity solution, $u$,  to \eqref{Eq-Intro} is pointwise differentiable and its gradient, $\nabla u$, is locally of class $C^{0, \min\{\alpha_0^{-}, \frac{1}{1+\gamma}\}}$, which is precisely the optimal regularity for degenerate equations of the type \eqref{Eq-Intro}. We further estimate the corresponding maximum regularity norm of $u$ by a constant that depends only on universal parameters, $\gamma$, $\|f\|_\infty$ and $\|u\|_\infty$. Sharpness of our estimate can be verified by simple examples. We have postponed the precise statement of the main Theorem to Section \ref{sec statement}. We highlight that the result proven in this manuscript is new even for the classical family of degenerate equations 
\begin{equation}\label{eq_Lap_intro}
	|\nabla u|^\gamma \Delta u = 1, \quad \gamma >0.
\end{equation}

\par

\medskip

The key, innovative feature of our main result lies precisely in the optimality of the gradient H\"older continuity exponent of a solution to the degenerate equation \eqref{Eq-Intro}, which in turn is an important piece of information in a number of qualitative analysis of PDEs, such as blow-up analysis, free boundary problems, geometric estimates, etc. It is quantitative bonus acquisition to the recent result in \cite{IS}, where it is proven that viscosity solutions to \eqref{Eq-Intro} are continuously differentiable. The logistic reasoning of the proof of our main result is inspired by recent works of the third author, \cite{T0}, \cite{T1}, \cite{T2}, and it further uses the main \textit{crack} from \cite{IS} to access  \textit{a priori} $C^1$ estimate for solutions to \eqref{Eq-Intro}.

\par

\medskip

The paper is organized as follows. In Section \ref{sec preliminar} we gather the most relevant notations and known results we shall use in the paper. In Section \ref{sec statement} we present the main Theorem proven in this work. In Section \ref{sec applications} we provide a few implications the sharp estimates from Section \ref{sec statement} have towards the solvability of some well known open problems in the elliptic regularity theory. The proof of Theorem \ref{Teo-Princ} is delivered in the remaining Sections \ref{sec comp}, \ref{sec flat}, \ref{sec small} and \ref{sec fim}.

\section{Notation and preliminares} \label{sec preliminar}

In this article we use standard notation from classical literature. The equations and problems studied in this paper are modeld in the $d$-dimensional Euclidean space, $\mathbb{R}^d$. The open ball of radius $r>0$ centered at the point $X_0$ is denoted by $B_r(X_0)$. Usually ball of radius $r$, centered at the origin is written simply as $B_r$. For a function $u\colon B_1 \to \mathbb{R}$, we denote its gradient and its Hessian at a point $X \in B_1$ respectively by 
$$
		\nabla u(X) :=(\partial_j u)_{1\le j\le  d} \quad \text{ and } \quad D^2u(X) := (\partial_{ij}u)_{1\le i, j \le d},
$$
where $\partial_j u$ and $\partial_{ij} u$ denote the $j$-th directional derivative of $u$ and the $i$-th directional derivative of $\partial_j u$, respectively.

The space of all $d\times d$ symmetric matrices is denoted by $\text{Sym}(d)$. An operator $F \colon B_{1} \times \text{Sym}(d) \rightarrow \mathbb{R}$  is said to be uniformly elliptic if there there exist two positive constants $0 < \lambda \le \Lambda$ such that, for any $M \in \text{Sym}(d)$ and $X \in B_1$,
\begin{equation}  \label{unif ellip}
    \lambda \|P\| \le F(X, M+P) - F(X, M) \le \Lambda \|P\|, \quad \forall P \ge 0.
\end{equation}

Any operator $F$ satisfying the ellipticity condition \eqref{unif ellip} will be referred hereafter in this paper as a  $(\lambda, \Lambda)$-elliptic operator. Also, following classical terminology, any constant or entity that depends only on dimension and the ellipticity parameters $\lambda$ and $\Lambda$ will be called \textit{universal}. For normalization purposes, we assume, with no loss of generality, throughout the text that $F(X, 0) = 0,  ~ \forall X \in B_1.$

For an operator $G\colon B_1 \times \mathbb{R}^d \times \text{Sym}(d) \to \mathbb{R}$, we say a function $u \in C^0(B_1)$ is a viscosity super-solution to $G(X, \nabla u, D^2u) = 0$, if whenever we touch the graph of $u$ by below at a point $Y \in B_1$ by a smooth function $\varphi$, there holds $G(Y, \nabla \varphi(Y), D^2 \varphi (Y)) \le  0$. We say $u \in C^0(B_1)$ is a viscosity sub-solution to $G(X, \nabla u, D^2u) = 0$, if whenever we touch the graph of $u$ by above at a point $Z \in B_1$ by a smooth function $\phi$, there holds $G(Y, \nabla \varphi(Y), D^2 \varphi (Y)) \ge  0$. We say $u$ is a viscosity solution if it is a viscosity super-solution and a viscosity sub-solution. The crucial observation  on the above definition is that if $G$ is non-decreasing on $M$ with respect to the partial order of symmetric matrices, then the classical notion of solution, sub-solution and super-solution is equivalent to the corresponding viscosity terms, provided the function is of class $C^2$. The theory of viscosity solutions to non-linear 2nd order PDEs is nowadays fairly well established. We refer the readers to the classical article \cite{user guide}. 

\par

Let us discuss now a little bit further the existing regularity theory for uniformly elliptic equations. As mentioned earlier in the Introduction, it follows from the celebrated Krylov-Safonov Harnack inequality, see for instance \cite{CC}, that any viscosity solution to the constant coefficient, homogeneous equation
\begin{equation}\label{hom eq}
	F(D^2h) = 0,
\end{equation}
is locally of class $C^{1,\alpha_0}$ for a universal exponent $0< \alpha_0 < 1$. Hereafter in this paper, $\alpha_0 = \alpha_0(d, \lambda, \Lambda)$ will always denote the optimal H\"older continuity exponent for solutions constant coefficients, homogeneous, $(\lambda, \Lambda)$-elliptic equation \eqref{hom eq}.
If no extra structural condition is imposed on $F$, $C^{1,\alpha_0}_\text{loc}$  is indeed the optimal regularity possible, \cite{NV1}, \cite{NV2}, \cite{NV}. However, under convexity or concavity assumption on $F$, solutions are of class $C^2$. This is a celebrated result due to Evans \cite{Ev} and Krylov \cite{K1, K2}, independently. 

\par

For varying coefficient equations, solutions are in general only $C_\text{loc}^{0, \alpha_0}$ and this is the optimal regularity available, unless we impose some continuity assumption on the coefficients, i.e., on the map $X \mapsto F(X, \cdot)$,  Such condition is quite natural and it is present even in the linear theory:  $Lu := a_{ij}(X) D_{ij} u$.  Since we aim for a universal $C^{1,\alpha}$ estimate for solutions to equation \eqref{Eq-Intro}, hereafter we shall assume a uniform continuity assumption on the coefficients of $F$, appearing in \eqref{Eq-Intro}, namely
\begin{equation}  \label{continuity}
	\sup\limits_{\|M\| \le 1} \frac{|F(X, M) - F(Y, M)|}{\|M\|} \le C \omega(|X-Y|),
\end{equation}
where $C\ge 0$ is a positive constant and $\omega$ is a normalized modulus of continuity, i.e., $\omega \colon \mathbb{R}_+ \to \mathbb{R}_+$ is increasing, $\omega(0^{+}) = 0$ and $\omega(1) = 1$.  Such condition could be relaxed: it suffices some sort of VMO condition, see \cite{C1}. We have decided to present the results of this present article under \eqref{continuity} for sake of simplicity. For notation convenience, we will call 
\begin{equation}\label{w-norm}
	\|F\|_\omega := \inf \left \{ C > 0 ~ : ~ \sup\limits_{\|M\| \le 1} \frac{|F(X, M) - F(Y, M)|}{\|M\|} \le C \omega(|X-Y|), ~ \forall X, Y \in B_1 \right \}.
\end{equation}

We close this Section by mentioning that under continuity condition on the coefficients, viscosity solutions to 
$$
	F(X, D^2u) = f(X) \in L^\infty(B_1)
$$
are locally of class $C^{1,\beta}$, for any $0< \beta < \alpha_0$, where $\alpha_0$ is the optimal H\"older exponent for solutions to constant coefficient, homogeneous equation $F(D^2h) = 0$, coming from Krylov-Safonov, Caffarelli universal regularity theory. See \cite{C1}, \cite{T1}. 
\par

\section{Main results} \label{sec statement}

In this Section, we shall present the main result we will prove in this present work. As mentioned earlier, the principal, ultimate goal of this article is to understand the sharp smoothness estimates for functions $u$, satisfying
\begin{equation}\label{Eq-mod02}
	|\nabla u|^{\delta}  \cdot |F(D^2u)| \lesssim 1, \quad \delta > 0,
\end{equation}
in viscosity sense, for some uniformly elliptic operator $F$. Clearly, as commented in the previous Section, even in the non-degenerate case, $\delta = 0$, the best regularity possible is $C_\text{loc}^{1+\alpha_0^{-}}$. The delicate point, though, is to obtain a universal estimate, fine enough as to sense and deem the singularity appearing in RHS of Equation \eqref{Eq-mod02}, along the singular set $(\nabla u)^{-1}(0)$, as $\delta>0$ varies.  

\par

As to grasp some feelings on what one should expect, let us na\"ively look at the ODE
$$
	u''(t) = (u')^{-\delta}, \quad u(0) = u'(0) = 0,
$$ 
which can be simply solved for $t \in (0, \infty)$. The solution is $u(t) = t^{\frac{2+\delta}{1+\delta}}$. After some heuristics inference, it becomes reasonable to accept that $C^{{\frac{2+\delta}{1+\delta}}}$ is another upper barrier for any universal regularity estimate for Equation \eqref{Eq-mod02}. Thus, if no further obscure complexity interferes on the elliptic regularity theory for fully nonlinear degenerate elliptic equation, the ideal optimal regularity estimate one should hope for functions satisfying Equation \eqref{Eq-mod02} should be $C^{1, \min\{\alpha_0^{-}, \frac{1}{1+\delta}\}}$. 

After these technical free and didactical comments, we are able to state the main result we establish in this paper.  
\begin{theorem}\label{Teo-Princ}
Let $u$ be a viscosity solutions to
\begin{equation}\label{(3)}
    \mathcal{H}(X,\nabla u)F(X,D^2u)=f(X) \quad \textrm{in} \quad B_1.
\end{equation}
Assume $f\in L^\infty(B_1)$,  $\mathcal{H}$ satisfies \eqref{(2)} and $F\colon B_1 \times \text{Sym}(d) \to \mathbb{R}$ is uniform elliptic with continuous coefficients, i.e., satisfying \eqref{continuity}. Fixed an exponent
$$
	\alpha \in (0, \alpha_0) \cap \left( 0, \frac{1}{1+\gamma}\right],
$$
there exists a constant $C(d, \lambda, \Lambda, \gamma, \|F\|_\omega,  \|f\|_\infty, \alpha)>0$, depending only on $d, \lambda, \Lambda, \gamma, \|F\|_\omega, \|f\|_\infty$ and $ \alpha$,  such that 
$$
    \|u\|_{C^{1,\alpha} (B_{1/2})} \le C(d, \lambda, \Lambda, \gamma, \|F\|_\omega, \|f\|_\infty, \alpha)  \cdot \|u\|_{L^{\infty}}.
$$
\end{theorem}

\par

An important consequence of Theorem \ref{Teo-Princ} is the following:

\begin{corollary} \label{cor} Let $u$ be a viscosity solutions to
\begin{equation}\label{(4)}
    \mathcal{H}(X,\nabla u)F(D^2u)=f(X) \quad \textrm{in} \quad B_1.
\end{equation}
Assume $f\in L^\infty(B_1)$, $\mathcal{H}$ satisfies \eqref{(2)}, $F$ is uniformly elliptic and concave.  Then $u$ is locally in $C^{1,\frac{1}{1+\gamma}}$ and this regularity is optimal.
\end{corollary}

Corollary \ref{cor} follows from Theorem \ref{Teo-Princ} since solutions to concave equations are locally of class $C^{1,1}$ by Evans-Krylov Theorem.

\par \medskip

It is interesting to understand Theorem \ref{Teo-Princ} as a model classification for degenerate elliptic equations, linking the magnitude of the degeneracy of the operator to the optimal regularity of solutions. A quantitative, intrinsic signature of the degeneracy properties of the equation. In fact, as mentioned earlier, many classical equations have their degree of degeneracy comparable to a model equation of the form $|\nabla u|^\gamma |F(D^2u)| \lesssim 1$. We shall explore this perspective within the next Section.

\section{Applications and further insights} \label{sec applications}

The heuristics from the ``degeneracy classification" mentioned in the previous paragraph has indeed a wide range of applicability.  In this intermediary Section we comment on some consequences the optimal regularity estimates stated in Section \ref{sec statement} have in the elliptic regularity theory.  

\par
\medskip

 In the sequel we shall use Theorem \ref{Teo-Princ} and Corollary \ref{cor} to solve particular cases of  some well known open problems. The results provided in this section give hope that decisive progress can be attempted for the general cases in the near future.

\subsection{Equations from the theory of superconductivity} 

\par
\medskip

We start off by commenting on some applications Theorem \ref{Teo-Princ} has to the theory of superconductivity, where fully nonlinear equations with patches of zero gradient
\begin{equation} \label{eq CS}
	F(X, D^2u)  =  g(X,u) \chi_{\{|\nabla u| > 0 \}}
\end{equation}
governs the mathematical models. Equation \eqref{eq CS} represents the stationary equation for the mean field theory of superconducting vortices  when the scalar stream function admits a functional dependence on the scalar magnetic potential, see \cite{Chap}. Existence and regularity properties of Equation \eqref{eq CS} were studied in \cite{CS} and in \cite{CSS}. The novelty to study Equation \eqref{eq CS} is that one tests the equation only for touching polynomials for which $|\nabla P(X_0)| \not = 0$. It is proven in \cite{CS}, Corollary 7, that solutions are locally $C^{0,\alpha}$ for some $0<\alpha<1$. For concave operators, it is proven, see \cite{CS} Corollary 8, that solutions are in $W^{2,p}$. An application of Alt-Caffarelli-Friedman monotonicity formula, \cite{CS} Lemma 9, gives regularity $C^{1,1}$ for the particular problem 
\begin{equation} \label{eq CS part}
	\Delta u = cu \ \chi_{\{|\nabla u| > 0 \}}.
\end{equation} 

\par
\medskip

Equation \eqref{eq CS} can be obtained as a limiting problem, as $\delta \to 0$, for the family of singular equations
\begin{equation} \tag{$E_\delta$} \label{eq CS delta}
	|\nabla u_\delta|^{\delta} \cdot F(X, D^2u_\delta)  =  g(X,u_\delta), \quad B_1.
\end{equation}
Indeed, it follows from Theorem \ref{Teo-Princ} that if $u_\delta$ is a normalized solution to \eqref{eq CS delta}, for $\delta$  small enough, i.e., for
$$
	\delta < 1 - \alpha_0^{-},
$$ 
then we can estimate
\begin{equation}\label{SC compact} 
	\|u\|_{C_\text{loc}^{1,\alpha_0^{-}}}ƒ \le C,
\end{equation}
for a constant $C>1$, that does not depend on $\delta$. In particular, estimate \eqref{SC compact} gives local compactness for the family of solutions $\{u_\delta\}_{\delta > 0}$ to \eqref{eq CS delta}. Let $u_0$ be a limiting point of such a sequence, i.e.
$$
	u_0 = \lim\limits_{j\to 0} u_{\delta_j},
$$
for  $\delta_j = \text{o}(1)$. From \eqref{SC compact}, we have,
\begin{eqnarray}
		\nabla u_{\delta_j} &\longrightarrow& \nabla u_0 \quad \text{locally uniformly,} \label{SC conv grad}\\
		u_0 &\in& C_\text{loc}^{1,\alpha_0^{-}}(B_1) \label{SC reg}.
\end{eqnarray}
Now, fixed a regular point $Z \in B_1$ of  $u_0$, i.e.,
$$
	|\nabla u_0(Z)| > 0.
$$
Gradient convergence \eqref{SC conv grad} and estimate  \eqref{SC reg} yield the existence of a small $\eta>0$, such that
$$
	\inf\limits_{B_\eta(Z)} |\nabla u_\delta| \ge \dfrac{1}{10} |\nabla u_0(Z)| =: c_0,
$$  
for all $\delta \ll 1$. Thus,
$$
	g(X, u_\delta) \cdot |\nabla u_\delta|^{-\delta} \longrightarrow g(X, u_0), \quad \text{ uniformly in } B_\eta(Z).
$$
We sum up the above discussion as the following Theorem:
\begin{theorem}\label{thm SC theory} Let $u_\delta \in C^{0}(B_1)$ be a viscosity solution to \eqref{eq CS delta}, with $|u_\delta|\le 1$, $\delta \ll 1$, where $g$ is continuous w.r.t. $u$ and measurable bounded w.r.t. $X$. Assume the operator $F$ is under the hypotheses of Theorem \ref{Teo-Princ}. Then, fixed a number $\alpha < \alpha_0(d, \lambda, \Lambda)$, for $\delta$ small enough, we have
$$
	\|u_\delta\|_{C^{1,\alpha}(B_{4/5})} \le C(d, \lambda, \Lambda, \alpha).
$$
In particular, 
$$
	u_\delta \to u_0 \in C^{1,\alpha}(B_{1/2}),
$$
and $u_0$ is a viscosity solution to \eqref{eq CS}.
\end{theorem}

The advantage of Theorem \ref{thm SC theory}, in comparison to the regularity theory developed in \cite{CS} is that it provides the asymptotically sharp $C^{1,\alpha}$ estimate in the general case of fully nonlinear operators, not necessarily concave.

\subsection{Visitng the theory of $\infty$-laplacian}

Let us now visit the theory of the $\infty$-laplacian operator, i.e., 
\begin{equation}\label{infty-lap}
		\displaystyle \Delta_\infty v :=  \displaystyle \sum\limits_{i,j} v_i v_j v_{ij}
\end{equation} 
which is related to the problem of best Lipschitz extension to a given boundary datum - a nonlinear and highly degenerate elliptic operator. The theory of infinity-harmonic functions, i.e., solutions to the homogeneous PDE
$$
	\Delta_\infty h = 0,
$$
has received a great deal of attention. One of the main open problems in the modern theory of PDEs is whether infinity-harmonic functions are of class $C^1$. This conjecture has been answered positively by O. Savin \cite{Savin} in the plane. Evans and Savin, \cite{Evans_Savin} sharpened the result to $C^{1,\alpha}$ for some small $ \alpha > 0$, but still in only dimension two. Quite recently, Evans and Smart proved that infinity-harmonic functions are everywhere differentiable regardless the dimension,  \cite{ES}. Nevertheless, no continuity feature of $\nabla u$ can be inferred by their ingenuous reasoning. The famous example of the infinity-harmonic function 
\begin{equation} \label{aronsson}
	a(x,y) := x^{\frac{4}{3}} - y^{\frac{4}{3}}
\end{equation}
due to Aronsson from the late 60's sets the ideal optimal regularity theory for such a problem. That is, no universal regularity theory for infinity harmonic functions can go beyond $C^{1,\frac{1}{3}}$. Up to our knowledge, there has been no prior meaningful mathematical indication that infinity-harmonic functions should or should not have a universal $C^{1,\frac{1}{3}}$ regularity theory, other than speculation based on Aronsson's example. Another way, though, to surmise the $C^{1,\frac{1}{3}}$ conjecture for infinity harmonic function would be by exploring the scaling properties of the equation. For instance, it one writes the infinity-laplacian as
$$
	\Delta_\infty v = (\nabla v)^t\cdot D^2 v \cdot \nabla v,
$$ 
it becomes tempting to compare its degeneracy feature with 
\begin{equation}\label{grad2hess}
	|\nabla u|^2 \cdot |\Delta u|   \lesssim 1,
\end{equation}
that has the same scaling properties as $\Delta_\infty$ and whose solutions are locally $C^{1,\frac{1}{1+2}}$ regular from Corollary \ref{cor}. Although, it is not in general true that infinity-harmonic functions satisfy \eqref{grad2hess}, this observation sets an interesting heuristic guide.

\par

\medskip

Notice that Aronsson's example - as many popular examples in the theory of PDEs - is a function of separable variables. In the sequel, as an application of Corollary \ref{cor}, we show that any infinity-harmonic function with separable variables is locally of class $C^{1, \frac{1}{3}}$. 

\begin{proposition}\label{infty-lap01} Let $u \colon B_1 \subset \mathbb{R}^d \to \mathbb{R}$ be infinity harmonic. Assume $u$ is a function of separable variables, i.e.,
$$
	u(X) = \sigma_1(x_1) + \sigma_2(x_2) + \cdots \sigma_d(x_d),
$$
for $\sigma_i \in C^0(B_1)$. Then $u \in C^{1, \frac{1}{3}}(B_{1/2})$. 
\end{proposition}

\begin{proof} Formal direct computation gives
\begin{equation} \label{eq01-proof-teo-infty-lap01}
	0 = \Delta_\infty u = |\sigma_1(x_1)'|^2 \sigma''(x_1) + |\sigma_1(x_2)'|^2 \sigma''(x_2) + \cdots + |\sigma_d(x_d)'|^2 \sigma''(x_d).
\end{equation}
It is a manner of routine to justify the above computation using the viscosity solution machinery. We notice, however, that the $i$th term in \eqref{eq01-proof-teo-infty-lap01} depends only upon the variable $x_i$. Thus, since they sum up to zero, each of them must be constant, i.e.,
\begin{equation} \label{eq01-proof-teo-infty-lap02}
	 |\sigma_i(x_i)'|^2 \sigma''(x_i) = \tau_i, \quad \sum\limits_{i=1}^d \tau_i = 0.
\end{equation}
$C^{1, \frac{1}{3}}$-regularity of each $\sigma_i$ follows from Corollary \ref{cor} and the proof of Proposition \ref{infty-lap01} is concluded.
\end{proof}

\par

\medskip

In a number of geometrical problems, it is often that solutions behave asymptotically radial near singular points. It is therefore interesting to analyze the regularity theory for solutions that are  smooth up to a possible radial singularity. More precisely, we say a function $u$ is smooth up to a possible radial singularity at a point $X_0$ if we can write, near $X_0$, 
$$
	u(X) = \varphi(X) + \psi(|X - X_0|),
$$
with $\varphi \in C^2$ and $\varphi(X) = \text{O}(|X-X_0|^2)$.

\par

\medskip

In the sequel we shall prove that functions smooth up to a possible radial singularity whose infinity-laplacian is bounded in the viscosity sense is of class $C^{1, \frac{1}{3}}$. This regularity is optimal as 
$$
	\Delta_\infty |X|^{\frac{4}{3}} = \text{cte}.
$$

\begin{theorem}\label{infty-lap02} Let $u \in C^0(B_1)$ satisfy 
$$
	\Delta_\infty u = f(X) \in L^\infty(B_1)
$$
in the viscosity sense. Assume $u$ is smooth up to a possible radial singularity. Then $u \in C^{1, \frac{1}{3}}_\text{loc}(B_1)$. 
\end{theorem}
\begin{proof} With no loss of generality, we can assume $X_0 = 0$. If $u = \varphi(X)+ \psi(|X|)$ is smooth up to a radial singularity near the origin, then formally a direct computation yields
\begin{eqnarray}
	\nabla  u (X) &=&  \nabla \varphi +  \psi' \frac{X}{|X|}, \nonumber \\
	D^2 u(X) &=& D^2\varphi + \dfrac{1}{|X|^2}\psi'' X \otimes X + \psi' \left [ \frac{1}{|X|} \text{Id} - \frac{1}{|X|^3} X \otimes X \right ]. \nonumber
\end{eqnarray}
Owing to the estimates
\begin{eqnarray}
	|X|^{-2} |\varphi| + |X|^{-1} |\nabla \varphi| + |D^2 \varphi| & \le& C_1, \nonumber \\
	|\psi| + |\nabla \psi| & \le&  C_2, \nonumber \\
	|\Delta_\infty u| & \le & C_3 \nonumber,
\end{eqnarray}
we end up with 
$$
	\left ( \text{O}(r) + | \psi'|^2 \right )\cdot |\psi''| \le 	C_4,
$$
which ultimately gives the desired regularity for $\psi$. Again is it standard to verify the above computation using the language of viscosity solutions. 
\end{proof}

For functions with bounded infinity-laplacian, E. Lindgren, following ideas from \cite{ES}, has recently established Lipschitz estimate and everywhere differentiability. 

\subsection{Further degenerate elliptic equations}

Another interesting example to visit is the $p$-laplacian operator, $p\ge 2$: 
\begin{equation}\label{p-laplace01}
	\Delta_p u := \text{div} \left ( |\nabla u|^{p-2} \nabla u \right ).
\end{equation}
It appears for instance as the Euler-Lagrangian equation associated to the $p$-energy integral
$$
	\int (Du)^p dX \to \text{min}.
$$
Equations involving the the $p$-laplacian operator has received a great deal of attention for the past fifty years or so. 
In particular, the regularity theory for $p$-harmonic functions has been an intense subject of investigation, since the late 60's, when Uraltseva in \cite{U} proved that weak solutions to the homogeneous $p$-laplacian equation
\begin{equation}\label{p-laplace02}
	\Delta_p h = 0,
\end{equation}
is locally of class $C^{1,\alpha(d,p)}$, for some $\alpha(d,p)>0$. The sharp regularity for $p$-harmonic functions in the plane was obtained by Iwanec and Manfredi, \cite{IM}. The precise optimal H\"older continuity exponent of the gradient of $p$-harmonic functions in higher dimensions, $d\ge 3$, has been a major open problem since then. 

\par
\medskip

The $p$-laplacian operator can be written in non-divergence form, simply by passing formally the derivatives through: 
\begin{equation}\label{p-laplace03}
	\Delta_p u =  |\nabla|^{p-2} \Delta u + (p-2)|\nabla u|^{p-4} \Delta_\infty u.
\end{equation}
The notions of weak solutions, using its divergence structure in \eqref{p-laplace01}, and the non-divergence form in \eqref{p-laplace03} are equivalent, \cite{JLM}.

\par
\medskip

Within the context of functions with bounded $p$-laplacian, the conjecture is that the optimal regularity should be $C^{p'}$, where
$$
	\frac{1}{p} + \frac{1}{p'} = 1.
$$
Our next result gives a partial answer to this conjecture.

\begin{theorem}\label{thm-p-lap01} Let $p \ge 2$ and  $u$ satisfy 
$$
	|\Delta_p u | \le C,
$$
Assume $u$ is smooth up to a possible radial singularity. Then $u \in C^{p'}_\text{loc}(B_1)$. 
\end{theorem}
\begin{proof} Assume $u = \varphi(X) + \psi(r)$, with $\varphi = \text{O}(r^2)$, $r = |X|$, has bounded $p$-laplacian in the viscosity sense. Direct computation implies
$$
	\left ( \text{O}(r) + |\psi'| \right )^{p-2} |\psi''| + \left ( \text{O}(r) + |\psi'| \right )^{p-4}|\psi'|^2 |\psi''| \le C + \tilde{C},
$$
holds in the viscosity sense, where $\tilde{C}$ depends only on $\varphi$, the non-singular part of $u$. Thus,  Corollary \ref{cor} gives 
$$
	u \in C_{\text{loc}}^{1, \frac{1}{1+(p-2)}} \cong C_{\text{loc}}^{p'},
$$
and the Theorem is proven.  
\end{proof}

Estimates of the form $|\nabla u|^{p-2} |D^2 u| < C$ are not rare in a number of geometric problems involving the $p$-laplacian operator, see for instance \cite{LS}. Let us also mention that estimates of the form $(\epsilon+ |\nabla v|^2)^{\frac{p-1}{2}}|D^2v| < C$ are usually obtained for bounded weak solutions to divergent form equations, $D_i\left (A^i(Dv) \right ) = 0$, see for instance \cite{G}, Chapter 8.

\par
\medskip

Let us finish this Section, by revisiting the derivation of the infinity-laplacian operator as the limit of $p$-laplace, as $p \to \infty$. Let $h \in C^0(B_1)$ be an infinity-harmonic function. For each $p \gg 1$, let $h_p$ be the solution to the boundary value problem
$$
	\left \{
		\begin{array}{rll}
			\Delta_p h_p &=& 0, \text{ in } B_{3/4} \\
			h_p & = & h, \text{ on } B_{3/4}.
		\end{array}
	\right.
$$ 
It is known that $h_p$ form a sequence of equicontinuous functions and $h_p \to h$ locally uniformly to $h$. In particular
$$
	\Delta_\infty h_p = \text{o}(1), \quad \text{as } p \to \infty.
$$ 
Hereafter, let us call $h_p$ the $p$-harmonic approximation of the infinity-harmonic function $h$ in $B_{3/4}$.

\begin{proposition} \label{prop1} Let $h \in C^0(B_1)$ be an infinity-harmonic function and $h_p$ its $p$-harmonic approximation. Assume $|\Delta_\infty h_p| = \text{O}(p^{-1})$ as $p \to \infty$. Then $h \in C^{1, \frac{1}{3}}(B_{1/2})$. 
\end{proposition}

\begin{proof} Since $h_p$ is $p$-harmonic, it satisfies 
$$
	|\nabla h_p|^{2} \Delta h_p = (2-p) \Delta_\infty h_p.
$$	
By the maximum principle,
$$
	\|h_p\|_{L^\infty_{(B_{4/5})}} \le \|h\|_{L^\infty_{(B_{1})}}.
$$
From the approximation hypothesis and Corollary \ref{cor}, we deduce
$$
	\|h_p\|_{C^{1,\frac{1}{3}}(B_{1/2})} \le C,
$$
for a constant $C$ that is independent of $p$.  The proof of Proposition follows by standard reasoning. 
\end{proof}

Another interesting Proposition regards $p$-harmonic functions with bounded infinity-laplacian.

\begin{proposition} \label{prop2} Let $u$ be a $p$-harmonic function in $B_1 \subset \mathbb{R}^d$. Assume $\Delta_\infty u \in L^\infty(B_1)$. Then $u \in C^{1, \frac{1}{3}}(B_{1/2})$. 
\end{proposition}

\begin{proof} The proof follows by similar reasoning as in the proof of Proposition \ref{prop1}. We omit the details.
\end{proof}

We leave as an open problem whether Proposition \ref{prop2} holds true without the extra assumption on the boundedness of the infinity-laplacian. We further conjecture that if $\alpha(d,p)$ is the optimal (universal) H\"older continuity exponent for $p$-harmonic functions, then
$$
	\alpha(d, p) >  \frac{1}{3} + \text{o}(1), \quad \text{ as p } \to \infty.
$$

\section{Universal compactness} \label{sec comp}

From this Section on we start delivering the proof for the main sharp regularity estimate announced in Section \ref{sec statement}, namely, Theorem \ref{Teo-Princ}. In this first step, we obtain a universal compactness device to access the optimal regularity theory for solutions to Equation \eqref{Eq-Intro}. The proof we shall present here uses the main technical tool obtained in the recent work of Imbert and Silvestre, \cite{IS}.

\begin{lemma} \label{univ-comp} Let $\q \in \mathbb{R}^d$ be an arbitrary vector and $u \in C(B_1)$, a viscosity solution to
\begin{equation} \tag{$E_{\q}$}
	|\q + \nabla u|^\gamma F(X, D^2u) = f(X),
\end{equation}
satisfying  $\|u\|_{L^\infty(B_1)} \le 1$. Given $\delta >0$, there exists $\varepsilon>0$, that depends only upon $d, \lambda, \Lambda$, and $\gamma$, such that if
\begin{equation}\label{uc-cond}
	\|M\|^{-1} \cdot \|F(X, M) - F(0,M)\|_{L^\infty(B_1)} +  \|f\|_{L^\infty(B_1)} < \varepsilon,
\end{equation}
then we can find a function $h$, solution to a constant coefficient, homogeneous, $(\lambda, \Lambda)$-uniform elliptic equation
\begin{equation}\label{uc-eq}
	\mathfrak{F}(D^2h) = 0, \quad B_{1/2} 
\end{equation}
such that
\begin{equation}\label{uc-conclusion}
	\|u - h\|_{L^\infty(B_{1/2})} \le \delta.
\end{equation}

\end{lemma}

\begin{proof}
Let us suppose, for the sake of contradiction, that the thesis of the Lemma fails. That means that we could find a number $\delta_0>0$ and sequences,  $F_j(X,M)$, $f_j$, $\q_j$ and $u_j$, satisfying
\begin{eqnarray}	
	  F_j(X,M) \text{ is } (\lambda, \Lambda)\text{-elliptic},  \label{proof-uc-eq01} \\
	\|M\|^{-1} \cdot \|F(X, M) - F(0,M)\|_{L^\infty(B_1)}  = \text{o}(1), \label{proof-uc-eq02} \\
	\|f_j\|_{L^\infty(B_1)} = \text{o}(1), \label{proof-uc-eq03} \\
	\|u_j\|_{L^\infty(B_1)} \le 1 \text{ and }|\q_j + \nabla u_j|^\gamma F_j(X, D^2u_j) = f_j,  \label{proof-uc-eq04}
\end{eqnarray}	
however,
\begin{equation}  \label{proof-uc-eq05}
	\sup\limits_{B_{1/2}} |u_j - h| \ge \delta_0,
\end{equation}
for any $h$ satisfying a constant coefficient, homogeneous, $(\lambda, \Lambda)$-uniform elliptic equation \eqref{uc-eq}.
\par
Initially, arguing as in \cite{IS},  the sequence $u_j$ is pre-compact in $C^0(B_{1/2})$-topology. In fact, as in \cite{IS},  Lemma 4, there is a universally large constant $A_0>0$, such that,  if for a subsequence $\{\q_{j_k}\}_{k\in \mathbb{N}}$, there holds,
$$
	|\q_{j_k}| \ge A_0, \quad \forall k \in \mathbb{N},
$$
then, the corresponding sequence of solutions, $\{u_{j_k}\}_{j\in \mathbb{N}}$, is bounded in $C^{0,1}(B_{2/3})$. If  
$$
	|\q_{j}| < A_0, \quad \forall j \ge j_0,
$$
then, by  Harnack inequality, see \cite{Imbert1},  $\{u_{j}\}_{j\ge j_0}$ is bounded in $C^{0,\beta}(B_{2/3})$ for some universal $0< \beta < 1$.   

From the compactness above mentioned,  up to a subsequence, $u_j \to u_\infty$ locally uniformly in $B_{2/3}$. Our ultimate goal is to prove that the limiting function $u_\infty$ is a solution to  a constant coefficient, homogeneous, $(\lambda, \Lambda)$-uniform elliptic equation \eqref{uc-eq}. For that we also divide our analysis in two cases.

If $|\q_j|$ bounded, we can extract a subsequence of $\{\q_j\}$, that converges to some $\q_\infty\in \mathbb{R}^d$. Also, by uniform ellipticity and \eqref{proof-uc-eq02}, up to a subsequence $F_j(X,\cdot)\to \mathfrak{F} (\cdot)$, and 
$$
	|\q_\infty + \nabla u_\infty|^\gamma  \mathfrak{F}(D^2u_\infty) = 0,
$$
Arguing as in \cite{IS}, Section $6$, we conclude that $u_\infty$ is a solution to a constant coefficient, homogeneous elliptic equation, which contradicts  \eqref{proof-uc-eq05}.

If $|\q_j|$ is unbounded, then taking a subsequence, if necessary,  $|\q_j| \to \infty$. In this case, define $\vec{e}_j = \q_j/|\q_j|$ and then $u_j$ satisfies
$$
	\left |  \vec{e}_{j}+\frac{\nabla u_{j}}{|\q_{j}|}\right |^\gamma F_{j}(X,D^2u_{j})=\frac{f_{j}(X)}{|\q_{j}|^\gamma}.
$$
Letting $j \to \infty$ and taking another subsequence, if necessary, we also end up with a limiting function $u_\infty$, satisfying 
$\mathfrak{F}_{\infty}(D^2 u_{\infty})=0$ for some $(\lambda, \Lambda)$-uniform elliptic operator, $\mathfrak{F}_{\infty}$. As before, this gives a contradiction to \eqref{proof-uc-eq05}. The Lemma is proven.
\end{proof}

\section{Universal flatness improvement} \label{sec flat}

In this Section, we deliver the core sharp oscillation decay that will ultimately imply the optimal $C^{1,\alpha}$ regularity estimate for solutions to Equation \eqref{Eq-Intro}. The first task is a step-one discrete version of the aimed optimal regularity estimate. This is the contents of next Lemma.

\begin{lemma} \label{key 1} Let $\q \in \mathbb{R}^d$ be an arbitrary vector and $u \in C(B_1)$  a normalized, i.e., $|u|\le 1$, viscosity solution to
\begin{equation} \tag{$E_{\q}$} 
	|\q + \nabla u|^\gamma F(X, D^2u) = f(X).
\end{equation}
Given $\alpha \in (0, \alpha_0) \cap (0, \frac{1}{\gamma + 1}]$, there exist constants $0< \rho_0 < 1/2$ and $\epsilon_0>0$, depending only upon $d, \lambda, \Lambda, \gamma$ and $\alpha$, such that if 
\begin{equation}\label{small-condition-lemma}
	\|M\|^{-1} \cdot \|F(X, M) - F(0,M)\|_{L^\infty(B_1)}  + \|f\|_{L^\infty(B_1)} \le \epsilon_0,
\end{equation}
then there exists an affine function $\ell(X) = a + \vec{b} \cdot X$, such that
$$
	\sup\limits_{B_{\rho_0}} |u(X) - \ell(X)| \le \rho_0^{1+\alpha}.
$$
Furthermore,
$$
	| a | + |\vec{b}| \le C(d, \lambda, \Lambda),
$$
for a universal constant $C(d, \lambda, \Lambda)$ that depends only upon dimension and ellipticity constants.
\end{lemma}
\begin{proof}
For a $\delta>0$ to be chosen {\it a posteriori}, let $h$ be a solution to a constant coefficient, homogeneous, $(\lambda, \Lambda)$-uniform elliptic equation that is $\delta$-close to $u$ in $L^\infty(B_{1/2})$. The existence of such a function is the thesis of Lemma \ref{univ-comp}, provided $\epsilon_0$ is chosen small enough, depending only on $\delta$ and universal parameters. Since our choice for $\delta$ - later in the proof - will depend only upon universal parameters, we will conclude that the choice of $\epsilon_0$ is too universal.

From normalization of $u$, it follows that $\|h\|_{L^\infty(B_{1/2})} \le 2$;  therefore, from the regularity theory available for $h$,  see for instance \cite{CC}, Chapters 4 and 5, we can estimate
\begin{eqnarray}
	\sup\limits_{B_r} |h(X) - \left( \nabla h(0) \cdot X + h(0) \right ) |& \le &C(d,\lambda, \Lambda) \cdot  r^{1+\alpha_0}  \quad \forall r>0, \label{uf-eq0} \\
 	|\nabla h(0)|   + |h(0)| &\le& C(d,\lambda, \Lambda), \label{uf-eq00}
\end{eqnarray}
for a universal constant $0<C(d,\lambda, \Lambda)$. Let us label 
\begin{equation} \label{uf-eq011}
	\ell(X) =  \nabla h(0) \cdot X + h(0).
\end{equation}
It readily follows from triangular inequality that 
\begin{equation} \label{uf-eq01}
		\sup\limits_{B_{\rho_0}} |u(X) - \ell(X)| \le  \delta + C(d,\lambda, \Lambda) \cdot \rho_0^{1+\alpha_0}.
\end{equation}
Now, fixed an exponent $\alpha < \alpha_0$, we select $\rho_0$ and $\delta$ as
\begin{eqnarray}	 
	\rho_0 &:=& \sqrt[\alpha_0 - \alpha]{ \frac{1}{2C(d,\lambda, \Lambda)}},  \label{uf-eq02} \\
	\delta &:=&  \dfrac{1}{2} \left ({ \frac{1}{2C(d,\lambda, \Lambda)}} \right )^{\frac{1+\alpha}{\alpha_0 - \alpha}},  \label{uf-eq03}
\end{eqnarray}
where $0< C(d,\lambda, \Lambda)$ is the universal constant appearing in \eqref{uf-eq0}. We highlight that the above choices depend only upon $d, \lambda, \Lambda$ and the fixed exponent $0< \alpha < \alpha_0$. Finally, combining \eqref{uf-eq0}, \eqref{uf-eq01}, \eqref{uf-eq02} and \eqref{uf-eq03}, we obtain
$$
	\begin{array}{lll} 
		\sup\limits_{B_{\rho_0}} |u(X) - \ell(X)| &\le&  \dfrac{1}{2} \left ({ \frac{1}{2C(d,\lambda, \Lambda)}} \right )^{\frac{1+\alpha}{\alpha_0 - \alpha}} + C(d, \lambda, \Lambda) \cdot \rho_0^{1+\alpha} \cdot  \rho_0^{ \alpha_0 - \alpha} \\
		& = &  \dfrac{1}{2} \rho_0^{1+\alpha}  + \dfrac{1}{2} \rho_0^{1+\alpha}  \\
		&=& \rho_0^{1+\alpha},
	\end{array}
$$
and the Lemma is proven.
\end{proof}

In the sequel, we shall iterate Lemma \ref{key 1} in appropriate dyadic balls as to obtain the precise sharp oscillation decay of the difference between $u$ and affine functions $\ell_k$. 
 \begin{lemma}\label{iterative}
Under the conditions of the previous lemma, there exists a sequence of affine functions $\ell_k(X):=a_{k}+\vec{b}_{k}\cdot X$ satisfying
\begin{equation}\label{coef-cond}
    |a_{k+1}-a_{k}| + \rho^{k}_{0}|\vec{b}_{k+1} - \vec{b}_{k}| \le C_0 \rho^{(1+\alpha)k}_{0},
\end{equation}
 such that
\begin{equation}\label{proof thm eq 01}
	\sup\limits_{B_{\rho_0^k}} | u(X) - \ell_k(X) | \le \rho_0^{k(1+\alpha)}.
\end{equation}
where $\alpha$ is a fixed exponent within the range
\begin{equation}\label{sharp}
    \alpha \in (0,\alpha_0) \cap \left(0,\frac{1}{1+\gamma}\right]
\end{equation}
and $C_0$ is a universal constant that depends only on dimension and ellipticity. 
\end{lemma}
\begin{proof}
We argue by finite induction. The case $k=1$ is precisely the statement of Lemma \ref{key 1}. Suppose we have verified \eqref{proof thm eq 01} for $j = 1, 2, \cdots, k$. Define the rescaled function
$$
	v(X) := \dfrac{(u-\ell_k)(\rho_0^kX)}{\rho_0^{k(1+\alpha)}}.
$$
It readily follows from the induction assumption that $|v| \le 1$. Furthermore, $v$ satisfies
$$
    \left | \rho^{-k \alpha} b_{k} + \nabla v \right |^\gamma F_{k}( X,D^2v)=f_{k}(X),
$$
where 
$$
	f_{k}(X) = \rho^{k[1-\alpha(1+\gamma)]}_{0}f(\rho^{k}_{0}X)
$$ 
and
$$
    F_{k}(X,M):= \rho^{k(1-\alpha)}_{0}F \left(\rho^{k}_{0}X, \frac{1}{\rho^{k(1-\alpha)}_{0}}M\right).
$$
It is standard to verify that the operator $F_{k}$ is $(\lambda,\Lambda)$-elliptic. Also, the $\omega$-norm of the corresponding coefficient oscillation of $F_k$, as defined on \eqref{w-norm}, hereafter called $\beta_k$, does not increase. Also, one easily estimate
\begin{eqnarray}
	\|f_{k}\|_{L^{\infty}(B_1)} &\le& \rho^{k[1-\alpha(1+\gamma)]}_{0} \|f\|_{L^{\infty}(B_{\rho^k_{0}})}\label{(12)}.
\end{eqnarray}
Due to the sharpness of the exponent selection made in \eqref{sharp}, namely $\alpha \le \frac{1}{1+\gamma}$, we conclude $(F_k, f_k)$ satisfies the smallness assumption \eqref{small-condition-lemma}, from Lemma  \ref{key 1}.

We have shown that $v$ is under the hypotheses of Lemma \ref{key 1}, which assures the existence in the affine function $\tilde{\ell}(X):= {a}+\vec{b} \cdot X$ with $| {a}|+|\vec{b}| \le C(d, \lambda, \Lambda)$, such that
\begin{equation}\label{(13)}
    \sup_{B_{\rho_{0}}}|v(X)-\tilde{\ell}(X)| \le \rho^{1+\alpha}_{0}.
\end{equation}
In the sequel, we define the $(k+1)$th approximating affine function,  $\ell_{k+1}(X):= a_{k+1} + \vec{b}_{k+1} \cdot X$, where the coefficients are given by
$$
    a_{k+1}:= a_{k} + \rho^{(1+\alpha)k}_{0}{a} \quad \text{ and } \quad \vec{b}_{k+1}:=\vec{b}_{k} + \rho^{\alpha k}_{0} \vec{b}.
$$
Rescaling estimate \eqref{(13)} back, we obtain
$$
    \sup_{B_{\rho^{k+1}_{0}}}|u(X)-\ell_{k+1}(X)| \le \rho^{(k+1)(1+\alpha)}_{0}
$$
and  the proof of Lemma \ref{iterative} is complete.
\end{proof}

\section{Smallness regime} \label{sec small}

In this Section we comment on the scaling features of the equation that allow us to reduce the proof of Theorem \ref{Teo-Princ} to the hypotheses of Lemma \ref{key 1} and Lemma \ref{iterative}. 
\par

Let $v \in C(B_1)$ be a viscosity solution to 
$$
	\mathcal{H}(X, \nabla v) F(X, D^2v) = f(X),
$$
where $\mathcal{H}$ satisfies \eqref{(2)} and $F$ is a $(\lambda, \Lambda)$-elliptic operator with continuous coefficients, i.e., satisfying \eqref{continuity}. Fix a point $Y_0 \in B_{1/2}$ define $u\colon B_1 \to \mathbb{R}$ as 
$$
	u(X) := \frac{v(\eta X + Y_0)}{\tau},
$$ 
for parameters $\eta$ and $\tau$ to be determined. We readily check that $u$ solves
$$
	\mathcal{H}_{\eta, \tau}(X, \nabla u) F_{\eta, \tau} (X, D^2u) = f_{\eta, \tau}(X),
$$
where
\begin{eqnarray}
	F_{\eta, \tau} (X, M) &:=& \dfrac{\tau}{\eta^2} F\left ( \eta X +Y_0, \dfrac{\eta^2}{\tau}M \right ) \label{Fmod} \\
	\mathcal{H}_{\eta, \tau}(X, \p)  &:=&  \left ( \dfrac{\tau}{\eta} \right )^{\gamma} H\left (\eta X + Y_0,  \dfrac{\eta}{\tau} \p \right ) \label{Hmod} \\
	f_{\eta, \tau}(X) &:=& \dfrac{\eta^{\gamma + 2}}{\tau^{\gamma + 1}} f(\eta X + Y_0).
\end{eqnarray}
Easily one verifies that $F_{\eta, \tau} $ is uniformly elliptic with the same ellipticity constants as the original operator $F$, i.e, it is another $(\lambda, \Lambda)$-elliptic operator. Also $\mathcal{H}_{\eta, \tau}$ satisfies the degeneracy condition \eqref{(2)}, with the same constants. Let us choose
$$
	\tau := \max \left \{ 1, \|v\|_{L^\infty(B_1)} \right \},
$$ 
thus, $|u|\le 1$ in $B_{1}$. Now, for the universal $\epsilon_0$ appearing in the statement of Lemma \ref{key 1}, choose
$$
	\eta := \min \left \{ 1, \lambda \cdot (\epsilon_0 \|f\|_{L^\infty}^{-1})^{\frac{1}{\gamma+2}}, \omega^{-1} \left (\frac{\epsilon_0}{C} \right ) \right \}.
$$
With these choices, $u$ is under the assumptions of Lemma \ref{key 1}. 

The above reasoning certifies that in order to show Theorem \ref{Teo-Princ}, it is enough to work under the smallness regime requested in the statement of Lemma \ref{key 1}. Once established the desired optimal regularity estimate the normalized function $u$, the corresponding estimate for $v$ follows readily.

\section{Sharp local regularity} \label{sec fim}

In this Section we conclude the proof of Theorem \ref{Teo-Princ}. From the conclusions delivered in Section \ref{sec small}, it suffices to show the aimed $C^{1,\alpha}$ estimate at the origin for a solution $u$ under the hypotheses of Lemma \ref{key 1} and Lemma \ref{iterative}.  For a fixed exponent $\alpha$ satisfying the sharp condition \eqref{sharp}, we will establish the existence of an affine function 
$$
	\ell_\star(X) := a_\star +\vec{b}_\star \cdot X,
$$
such that
$$
	|\vec{b}_\star| + |a_\star| \le C,
$$
and 
$$
	\sup\limits_{B_r} \left | u(X) - \ell_\star(X) \right | \le C r^{1+\alpha}, \quad \forall r \ll 1,
$$
for a constant $C$ that depends only on $d$, $\lambda$, $\Lambda$, $\gamma$ and $\alpha$. 

\par

\medskip

Initially, we notice that it follows from \eqref{coef-cond} that the coefficients  of the sequence of affine functions $\ell_k$ generated in Lemma \ref{iterative}, namely  $\vec{b}_k$ and $a_k $, are Cauchy sequences in $\mathbb{R}^d$ and in $\mathbb{R}$, respectively. Let $\vec{b}_\star$ and $a_\star$ be the limiting coefficients, i.e., 
\begin{eqnarray} \label{f-p eq01}
	\lim\limits_{k\to \infty} \vec{b}_k &=:& \vec{b}_\star \in \mathbb{R}^d  \\
	 \label{f-p eq02}\lim\limits_{k\to \infty} a_k &=: & a_\star \in \mathbb{R}.
\end{eqnarray}
It also follows from the estimate obtained in \eqref{coef-cond} that
\begin{eqnarray}
	|a_\star - a_k| &\le&  \dfrac{C_0}{1-\rho_0}  \rho_0^{k(1+\alpha)},  \label{f-p eq03} \\
	|\vec{b}_\star - \vec{b}_k|  &\le&  \dfrac{C_0}{1-\rho_0}  \rho_0^{k\alpha}.  \label{f-p eq04}
\end{eqnarray}

\par

\medskip

Now, fixed a $0< r < \rho_0$, we choose $k \in \mathbb{N}$ such that
$$
	\rho_0^{k+1} < r \le \rho_0^k.
$$

We estimate
$$
	\begin{array}{lll}
		\sup\limits_{B_r} \left | u(X) - \ell_\star(X) \right | &\le& \sup\limits_{B_{\rho^k_0}} \left | u(X) - \ell_\star(X) \right | \\
		&\le &  \sup\limits_{B_{\rho^k_0}} \left | u(X) - \ell_k (X) \right | +  \sup\limits_{B_{\rho^k_0}} \left | \ell_k(X) - \ell_\star(X) \right | \\
		&\le & \rho_0^{k(1+\alpha)} + \dfrac{C_0}{1-\rho_0}  \rho_0^{k(1+\alpha)} \\
		&\le & \dfrac{1}{\rho_0^{1+\alpha}} \left [1 +  \dfrac{C_0}{1-\rho_0} \right ] \cdot r^{1+\alpha},
	\end{array}
$$
and the proof of Theorem \ref{Teo-Princ} is finally complete. \hfill $\square$

\section*{Acknowledgement} 

The authors would like to thank Luis Silvestre for insightful and stimulating comments and suggestions that benefited a lot the final presentation of this article.  This work has been partially funded by CNPq-Brazil.

\vspace{1cm}
\noindent \textsc{Dami\~ao Ara\'ujo} \hfill \textsc{Gleydson C. Ricarte}\\
Universidade Federal Cear\'a\hfill  Universidade Federal Cear\'a \\
Department of Mathematics \hfill Department of Mathematics \\
Fortaleza, CE-Brazil 60455-760 \hfill Fortaleza, CE-Brazil 60455-760\\
\texttt{djunio@mat.ufc.br} \hfill
\texttt{gleydsoncr@ufc.br}

\vspace{1cm}
\noindent \textsc{Eduardo V. Teixeira}\\
Universidade Federal Cear\'a \\
Department of Mathematics  \\
Fortaleza, CE-Brazil 60455-760 \\
\texttt{teixeira@mat.ufc.br}
\end{document}